\documentclass[11pt]{article}

\usepackage[margin=1.15in]{geometry}
\usepackage{amsmath,amssymb,amsthm}
\usepackage[colorlinks=true,linkcolor=blue,citecolor=blue,urlcolor=blue]{hyperref}

\newtheorem{theorem}{Theorem}[section]
\newtheorem{lemma}[theorem]{Lemma}
\newtheorem{corollary}[theorem]{Corollary}
\newtheorem{definition}[theorem]{Definition}
\newtheorem{problem}[theorem]{Problem}

\newcommand{\blue}{\mathrm{blue}}

\title{Ordered Ramsey numbers of 3-uniform hypergraphs with bounded weak degeneracy}
\author{
Wen Chen$^1$\thanks{
Email: chenwenfj@163.com },\quad \quad
Zihan He$^1$\thanks{
Email: 3213223910@qq.com},\quad \quad
Qizhong Lin$^1$\thanks{
Email: linqizhong@fzu.edu.cn. Supported in part by the National Key R\&D Program of China (Grant No. 2023YFA1010202) and the NSFC (No.\ 12571361).},\quad \quad
Meng Liu$^2$\thanks{
Email: liumeng@ahu.edu.cn, Supported in part by NSFC (12331012,12471323)}
\\
\small $^1${Center for Discrete Mathematics, Fuzhou University, Fujian 350108, China}\\
\small $^2${Center of pure Mathematics, School of Mathematical Sciences, Anhui University, Hefei, Anhui 230601, China}
}
\date{}
\begin{document}
\maketitle

\begin{abstract}

The \emph{ordered Ramsey number} $r_<(G,H)$ of ordered $k$-graphs $G$ and $H$ is the least integer $N$ such that every red-blue edge-coloring of the naturally ordered complete $k$-graph on $[N]$ contains a blue ordered copy of $G$ or a red ordered copy of $H$.
We prove that there is an absolute constant $c>0$ such that, for every
integer $d\ge1$, there is a constant $C_d>0$ for which every weakly
$d$-degenerate ordered $3$-graph $H$ on $t$ vertices satisfies
\[
r_<\bigl(H,K_3^{(3)}(n)\bigr)
\le t\,2^{C_d n^{2-c/d}}
\]
for every positive integer $n$. This resolves a problem posed by Balko and Vizer ({\em SIAM J. Discrete Math., 2022}) in a stronger form.

Furthermore, we show that the weak-degeneracy hypothesis cannot be replaced
by bounded standard degeneracy. In particular, for every sufficiently large
$n$, there exists a $1$-degenerate ordered $3$-graph $F$ on at most
$2^{O(n)}$ vertices such that
$r_<\bigl(F,K_3^{(3)}(n)\bigr)>2^{\Omega(n^2)}.$

\end{abstract}

\section{Introduction}

An \emph{ordered $k$-uniform hypergraph}, or ordered $k$-graph, is a $k$-graph equipped with a linear order on its vertex set. An ordered $k$-graph $G$ contains an ordered copy of an ordered $k$-graph $H$ if there is an order-preserving injection $\phi\colon V(H)\to V(G)$ such that $\phi(e)\in E(G)$ for every $e\in E(H)$.

For ordered $k$-graphs $G$ and $H$, the \emph{ordered Ramsey number} $r_<(G,H)$ is the smallest integer $N$ such that every red--blue coloring of the edges of the naturally ordered complete $k$-graph on $[N]$ contains a blue ordered copy of $G$ or a red ordered copy of $H$. In the diagonal case, we write $r_<(H)=r_<(H,H)$.

The origins of ordered Ramsey theory can be traced back to the classical work of Erd\H{o}s and Szekeres~\cite{erdosSzekeres35}, while its systematic study was initiated independently by Conlon, Fox, Lee and Sudakov~\cite{clfs17} and by Balko, Cibulka, Kr\'al and Kyn\v{c}l~\cite{bckk13}; see also the surveys~\cite{cfsSurvey,balko25}. One motivation comes from discrete geometry, and another from the striking differences between ordered and unordered Ramsey numbers, particularly for sparse graphs~\cite{bckk13,bjv16,clfs17}.
Ordered Ramsey theory for hypergraphs of higher uniformity, however, remains far less developed. It is therefore natural to seek sharper upper bounds for broader classes of sparse ordered $3$-graphs.

In the unordered setting, every $n$-vertex graph of bounded maximum degree has Ramsey number linear in $n$. By contrast, ordered matchings can have superpolynomial ordered Ramsey numbers~\cite{bckk13,clfs17}. The phenomenon is even stronger for ordered hypergraphs. Building on earlier work of Fox, Pach, Sudakov and Suk~\cite{fpss12}, Moshkovitz and Shapira~\cite{moshShap14} proved that, for every fixed $k\ge 3$, the ordered Ramsey numbers of monotone $k$-uniform tight paths have tower growth, despite these paths having bounded maximum degree. However, in the unordered setting, for every fixed $k$ and $d$, every $n$-vertex $k$-graph of maximum degree at most $d$ has Ramsey number $O_{k,d}(n)$~\cite{crst83,cfs09,cnko08,cnko09,nsrs08}.

This naturally leads to the search for structural parameters that take the vertex order into account. One such parameter is the interval chromatic number.
 A vertex set is an \emph{interval} if it consists of consecutive vertices in the underlying linear order.
The \emph{interval chromatic number} $\chi_<(H)$ of an ordered $k$-graph
$H$ is the minimum number of intervals into which $V(H)$ can be partitioned
so that every edge meets each interval in at most one vertex.
Let $K_\chi^{(k)}(n)$ denote the ordered complete $\chi$-partite $k$-graph whose vertex classes are consecutive intervals of size $n$, arranged from left to right.
For ordered graphs, Conlon, Fox, Lee and Sudakov~\cite{clfs17} proved that if $G$ is a $d$-degenerate ordered graph on $n$ vertices with interval chromatic number $\chi$, then
\begin{equation}\label{eq:ordered-graph-bound}
r_<\bigl(G,K_\chi^{(2)}(n)\bigr)\le n^{32d\log\chi}.
\end{equation}
Thus, bounded degeneracy together with bounded interval chromatic number guarantees a polynomial ordered Ramsey bound for graphs.

The situation for ordered $3$-graphs is less well understood. For every fixed interval chromatic number $\chi$,  a result of Conlon, Fox and Sudakov~\cite{cfs11} implies that every $n$-vertex ordered $3$-graph $H$ with $\chi_<(H)\le \chi$ satisfies
\[
r_<(H)\le 2^{O_\chi(n^2)}.
\]
This bound is asymptotically tight for  complete $\chi$-partite $3$-graphs. It is therefore natural to ask whether the exponent can be improved when $H$ is sparse. Balko and Vizer \cite[Theorem~2]{BV} showed that this is indeed possible under a bounded maximum-degree assumption, even in the stronger off-diagonal setting.

\begin{theorem}[Balko, Vizer \cite{BV}]
\label{thm:bv-upper}
Let $H$ be an ordered $3$-graph on $t$ vertices with maximum degree at most $d$, and let $s$ be a positive integer. Then there exist constants $C=C(d)>0$ and $c>0$ such that
\[
r_<\bigl(H,K_3^{(3)}(s)\bigr)
\le t2^{C s^{2-1/(1+cd^2)}}.
\]
\end{theorem}

On the other hand, a result of Fox and He~\cite{foxHe} implies that
$r_<\bigl(K_4^{(3)},K_3^{(3)}(s)\bigr)
\ge 2^{\Omega(s\log s)}$.
Balko and Vizer~\cite{BV} also gave a direct proof of this lower bound.
A substantial gap between the known upper and lower bounds remains.

\begin{definition}[Weak degeneracy]\label{def}
An ordered $k$-graph $H$  on $t$ vertices is \emph{weakly $d$-degenerate} if its underlying $k$-graph admits an ordering $u_1,\ldots,u_t$ of its vertices such that, for every $i$,
\[
\bigl|\{e\in E(H): u_i\in e
\text{ and }
e\cap\{u_1,\ldots,u_{i-1}\}\ne\varnothing\}\bigr|
\le d.
\]
\end{definition}
Note that Balko and Vizer \cite{BV} call this notion simply degeneracy. Here we call it weak degeneracy to distinguish it from a more commonly used notion of hypergraph degeneracy, which we refer to below as standard degeneracy.

Motivated by the corresponding results for ordered graphs of bounded degeneracy and bounded interval chromatic number~\cite{clfs17}, Balko and Vizer~\cite[Section~6]{BV} asked whether the bounded maximum degree assumption in Theorem~\ref{thm:bv-upper} could be replaced by bounded weak degeneracy.

\begin{problem}[Balko, Vizer \cite{BV}]
\label{prob:bv-degeneracy}
Can Theorem~\ref{thm:bv-upper} be extended from ordered $3$-graphs of bounded maximum degree to
ordered $3$-graphs of bounded weak degeneracy?
\end{problem}
Our main theorem answers Problem~\ref{prob:bv-degeneracy} affirmatively.
Moreover, it improves the dependence on $d$ in the exponent of Theorem~\ref{thm:bv-upper}.

\begin{theorem}\label{thm:main}
There is an absolute constant $c>0$ such that, for every integer $d\ge1$,
there is a constant $C_d>0$ with the following property. Let $H$ be a weakly $d$-degenerate ordered $3$-graph on $t$
vertices. Then, for every positive integer $n$,
\[
  r_<(H,K^{(3)}_3(n))\le t\,2^{C_d n^{2-c/d}} .
\]
\end{theorem}

For simplicity, we make no attempt to optimize the absolute constant $c$. Our proof builds on the embedding method of
Conlon, Fox and Sudakov~\cite{CFS}, as further developed by Balko and
Vizer~\cite{BV}. The main new ingredient is an embedding scheme tailored
to weak degeneracy. By separately tracking edges with one and two
previously embedded vertices and maintaining suitably chosen candidate
sets and auxiliary pair graphs, we control the embedding process using
only the weak-degeneracy condition, even though the maximum degree may be
unbounded. As in Balko and Vizer's work, the off-diagonal theorem
immediately yields a diagonal consequence for ordered $3$-graphs of
interval chromatic number at most three.

\begin{corollary}\label{cor:main-diagonal}
There is an absolute constant $c>0$ such that, for every integer $d\ge1$,
there is a constant $C_d>0$ for which every weakly $d$-degenerate ordered
$3$-graph $H$ on $n$ vertices with $\chi_<(H)\le3$ satisfies
\[
  r_<(H)
  \le \,2^{C_d n^{2-c/d}}.
\]
\end{corollary}

\begin{proof}
Since $\chi_<(H)\le3$, the ordered $3$-graph $H$ is contained in
$K_3^{(3)}(n)$, so we obtain that $r_<(H)\le\allowbreak r_<(H,K_3^{(3)}(n))$.
The result follows from Theorem~\ref{thm:main} by absorbing the factor $n$
into the exponential.
\end{proof}

We next examine whether the weak-degeneracy hypothesis in
Theorem~\ref{thm:main} is essential. For this purpose, we recall the
standard notion of hypergraph degeneracy.

\begin{definition}[Standard degeneracy]\label{def2}
An ordered $k$-graph $H$ on $t$ vertices is \emph{$d$-degenerate} if its underlying $k$-graph admits an ordering $u_1,\ldots,u_t$ of its vertices such that, for every $i$,
\[
\bigl|\{e\in E(H): u_i\in e
\text{ and }
e\setminus\{u_i\}\subseteq\{u_1,\ldots,u_{i-1}\}\}\bigr|
\le d.
\]
\end{definition}

Note that weak $d$-degeneracy implies standard $d$-degeneracy. Although
the two notions coincide for graphs, standard degeneracy is strictly
weaker for hypergraphs of uniformity at least three.

Our second result shows that the weak-degeneracy hypothesis cannot be
replaced by standard hypergraph degeneracy.

\begin{theorem}\label{main-2}
There are absolute constants $c,C>0$ such that, for every sufficiently
large $n$, there exists a $1$-degenerate ordered $3$-graph $F$ on at most
$2^{Cn}$ vertices such that
\[
  r_<\bigl(F,K_3^{(3)}(n)\bigr)>2^{c n^2}.
\]
\end{theorem}

The construction underlying Theorem~\ref{main-2} is obtained by lifting a
suitably biased random coloring of pairs to a coloring of triples. The
auxiliary pair-coloring is chosen to contain neither a red $K_{n,n}$ nor
a blue $K_m$. We then construct $F$ by introducing, for each pair of
vertices of an ordered $K_m$, a distinct new vertex completing that pair
to an edge. This makes $F$ standard $1$-degenerate. We place all new vertices after the original
vertices and colour each triple according to its first two vertices.
Under the lifted
coloring, a red ordered $K_3^{(3)}(n)$ would yield a red $K_{n,n}$ in the
auxiliary coloring, whereas a blue ordered copy of $F$ would yield a
blue $K_m$. These two obstructions give the claimed lower bound.

Since $F$ has only $2^{O(n)}$ vertices, Theorem~\ref{main-2} shows that no
analogue of Theorem~\ref{thm:main} with a subquadratic exponent in $n$ can
hold for ordered $3$-graphs of bounded standard degeneracy.

The remainder of the paper is organized as follows. In Section~2, we collect
the preliminary tools needed for the proof of our main result. In Section~3,
we prove Theorem~\ref{thm:main} by means of an iterative embedding argument.
In Section~4, we prove Theorem~\ref{main-2} and show that standard hypergraph
degeneracy is not sufficient for an analogue of our main upper bound.

\section{Preliminaries}
Throughout the paper, all logarithms are to base $2$ and are denoted by $\log$. We omit floor and ceiling signs whenever they do not affect the arguments.

For disjoint sets $X,Y$ and a bipartite graph $P\subseteq X\times Y$, write
\[
  d_P(X',Y')=\frac{e_P(X',Y')}{|X'||Y'|}
\]
for non-empty $X'\subseteq X$, $Y'\subseteq Y$.

\begin{definition}
Let $0<\alpha,\beta,\sigma\le 1$. A bipartite graph $P\subseteq X\times Y$ is
bi-$(\alpha,\beta,\sigma)$-dense if every $X'\subseteq X$, $Y'\subseteq Y$ with
\[
  |X'|\ge \alpha |X|,\qquad |Y'|\ge \beta |Y|
\]
satisfies $d_P(X',Y')\ge \sigma$.
\end{definition}

The following embedding lemma shows that if a tripartite graph contains sufficiently many triangles and almost all of them are edges of a $3$-graph $R$, then $R$ contains a large complete tripartite $3$-graph.

\begin{lemma}[Conlon, Fox and Sudakov \cite{CFS}]\label{lem:cfs}
There is an absolute constant $c_0>0$ such that the following holds. Let
$m$ and $s$ be positive integers, and let $0<\delta\le1$. Let $V_1,V_2,V_3$ be pairwise disjoint sets, each of size at most $m$, and let
$G_{ij}$ be a bipartite graph between $V_i$ and $V_j$ for $1\le i<j\le 3$.
Suppose that the tripartite graph $G_{12}\cup G_{13}\cup G_{23}$ has at least
$\delta m^3$ triangles. Let $R$ be a 3-graph containing at least a
$(1-\eta)$-proportion of these triangles, where $0<\eta<1/8$. If
\[
  \exp(c_0\delta^{-2}s^{3/2})(1-4\eta)^{-4s^2}
  \left(\frac{16}{\delta}\right)^{4s}\le m,
\]
then $R$ contains a copy of $K^{(3)}_3(s)$ with one vertex class in each $V_i$.
\end{lemma}

The exact value of $c_0$ is irrelevant in what follows. In particular, after replacing $\delta$ by $\delta/2$ and $\eta$ by
$2\rho$, the size condition in Lemma~\ref{lem:cfs} is guaranteed whenever
\[
  m\ge
  \exp\left(
    C\left(
      \delta^{-2}s^{3/2}
      +s^2\log\frac{1}{1-8\rho}
      +s\log\frac{1}{\delta}
    \right)
  \right),
\]
where $C>0$ is an absolute constant.

\begin{lemma}\label{lem:sampling}
Let $B,C$ be finite sets and let $m$ be a positive integer with
$m\le \min\{|B|,|C|\}$. Let $w,u:B\times C\to [0,\infty)$ be weight functions.
Put
\[
  W=\sum_{b\in B,\ c\in C} w(b,c),\qquad
  U=\sum_{b\in B,\ c\in C} u(b,c).
\]
Assume that $W\ge \delta m|B||C|$,  $U\le \theta W$,
where $0<\delta,\theta\le 1$. Then there are subsets
$B'\subseteq B$ and $C'\subseteq C$ such that $|B'|=|C'|=m$ and
\[
  \sum_{(b,c)\in B'\times C'} w(b,c)\ge \frac12\delta m^3,
\qquad
  \sum_{(b,c)\in B'\times C'} u(b,c)
  \le 2\theta
  \sum_{(b,c)\in B'\times C'} w(b,c).
\]
\end{lemma}

\begin{proof}
Choose $B'$ uniformly at random from $\binom{B}{m}$ and $C'$ uniformly at
random from $\binom{C}{m}$, independently. Put
\[
  Z=\sum_{(b,c)\in B'\times C'} w(b,c),
  \qquad
  Y=\sum_{(b,c)\in B'\times C'} u(b,c).
\]
Each fixed pair $(b,c)\in B\times C$ belongs to $B'\times C'$ with probability $\frac{m^2}{|B||C|}$.
Hence
\[
  \mathbb E Z
  =
  \frac{m^2}{|B||C|}W
  \ge
  \delta m^3,
\qquad
  \mathbb E Y
  =
  \frac{m^2}{|B||C|}U
  \le
  \theta\,\mathbb E Z.
\]
Therefore
\[
  \mathbb E\left[Z-\frac{1}{2\theta}Y\right]
  =
  \mathbb E Z-\frac{1}{2\theta}\mathbb E Y
  \ge
  \frac12\,\mathbb E Z
  \ge
  \frac12\delta m^3.
\]
Thus there is a choice of $B'$ and $C'$ for which
  $Z-\frac{1}{2\theta}Y  \ge \frac12\delta m^3$.
For this choice, since $Y\ge 0$, we immediately get
  $Z\ge \frac12\delta m^3$.
Moreover, $Z-\frac{1}{2\theta}Y>0$,
and hence  $Y<2\theta Z$.
Therefore
\[
  \sum_{(b,c)\in B'\times C'} w(b,c)\ge \frac12\delta m^3,
\qquad
  \sum_{(b,c)\in B'\times C'} u(b,c)
  \le
  2\theta
  \sum_{(b,c)\in B'\times C'} w(b,c),
\]
as required.
\end{proof}

Given a red-blue coloring of the triples of a vertex set $V$, for a vertex
$a\in V$, the \textbf{blue link graph} $L_{\blue}(a)$ is the graph on
$V\setminus\{a\}$ in which $bc$ is an edge if and only if $abc$ is a blue
triple.

The following lemma is where we use the assumption that there is no red ordered
$K^{(3)}_3(s)$.
Let $0<\rho<1$, let $A,B,C$ be pairwise disjoint subsets of an ordered vertex set whose triples
are red-blue coloured, and let $P\subseteq B\times C$ be a bipartite graph.
Suppose that $P$ is bi-$(\eta_B,\eta_C,\sigma)$-dense. A vertex $a\in A$ is
called \textbf{bad} for $(B,C,P)$ if there exist subsets $B_a\subseteq B$ and
$C_a\subseteq C$ such that $|B_a|\ge \eta_B|B|$, $|C_a|\ge \eta_C|C|$,
but
\[
  e_{P\cap L_{\blue}(a)}(B_a,C_a)
  < \rho\sigma |B_a||C_a|,
\]
where $L_{\blue}(a)$ denotes the blue link graph of $a$ on $B\times C$.

\begin{lemma}\label{lem:bad}
There is an absolute constant $C_2$ such that the following holds. Let
$N$ and $s$ be positive integers, let $0<\rho<1/16$, and let
$A,B,C\subseteq [N]$ lie in three pairwise disjoint
intervals. Let $P\subseteq B\times C$ be bi-$(\eta_B,\eta_C,\sigma)$-dense,
where $0<\eta_B,\eta_C,\sigma\le 1$. Suppose that a red-blue coloring of the
triples on $[N]$ contains no red ordered copy of $K^{(3)}_3(s)$.
Set  $\delta=\sigma\eta_B\eta_C$
and
\[
  M=
  2^{C_2\left(\delta^{-2}s^{3/2}
  +s^2\log\frac{1}{1-8\rho}
  +s\log\frac{1}{\delta}\right)}.
\]
If $|B|,|C|\ge M$, then fewer than $M$ vertices of $A$ are bad for
$(B,C,P)$.
\end{lemma}

\begin{proof}
Suppose, for contradiction, that there exists a vertex set $X\subseteq A$ of bad vertices for $(B,C,P)$ with $|X|=M$.
For each $a\in X$, take $B_a\subseteq B$ and $C_a\subseteq C$ satisfying $|B_a|\ge \eta_B |B|$, $|C_a|\ge \eta_C |C|$,
but
$$e_{P\cap L_{\blue}(a)}(B_a,C_a)<\rho\sigma |B_a||C_a|.$$
Since $P$ is bi-$(\eta_B,\eta_C,\sigma)$-dense, we have
$e_P(B_a,C_a)\ge \sigma |B_a||C_a|$.

Define two weight functions on $B\times C$ by
$$ w(b,c)=|\{a\in X:b\in B_a,\ c\in C_a,\ bc\in P\}|$$
and
$$u(b,c)=|\{a\in X:b\in B_a,\ c\in C_a,\ bc\in P,\ abc\text{ is blue}\}| .$$
Clearly $w,u\ge 0$. Since $P$ is
bi-$(\eta_B,\eta_C,\sigma)$-dense,
\[
  \sum_{b,c} w(b,c)
  =\sum_{a\in X} e_P(B_a,C_a)
  \ge M\sigma\eta_B\eta_C |B||C|
  =\delta M|B||C|.
\]
By the definition of bad vertices,
\[
  \sum_{b,c} u(b,c)
  =\sum_{a\in X} e_{P\cap L_{\blue}(a)}(B_a,C_a)
  <\rho\sum_{a\in X}\sigma |B_a||C_a|
  \le \rho\sum_{a\in X} e_P(B_a,C_a)
  =\rho\sum_{b,c} w(b,c).
\]
Apply Lemma~\ref{lem:sampling} with $m=M$ and $\theta=\rho$. Thus there are
$B'\subseteq B$ and $C'\subseteq C$ such that $|B'|=|C'|=M$,
\[
  Z:=\sum_{(b,c)\in B'\times C'}w(b,c)\ge \frac12\delta M^3,
\qquad
  Y:=\sum_{(b,c)\in B'\times C'}u(b,c)\le 2\rho Z .
\]

Construct a tripartite graph with vertex classes $X,B',C'$ and bipartite edge
sets
\[
  G_{XB'}=\{ab:a\in X,\ b\in B'\cap B_a\},
\qquad
  G_{XC'}=\{ac:a\in X,\ c\in C'\cap C_a\},
\]
and $G_{B'C'}=P[B',C']$.
For $a\in X$, $b\in B'$, and $c\in C'$, the triple $(a,b,c)$ forms a triangle
in this tripartite graph if and only if
  $b\in B_a$, $c\in C_a$ and $bc\in P$.
Hence the number of triangles is exactly
 $\sum_{(b,c)\in B'\times C'} w(b,c)=Z$ and the number of blue triangles is at most $Y$. Therefore at least a $(1-2\rho)$-proportion of these triangles are
red.

We verify the hypotheses of Lemma~\ref{lem:cfs} with
$m=M$, $\delta'=\delta/2$ and $\eta=2\rho$.
Each vertex class has size $M$, the number of triangles is
$Z\ge\delta'M^3$, and at least a $(1-\eta)$-proportion of these
triangles are red. Moreover, $0<\eta<1/8$ since $0<\rho<1/16$,
and the size condition follows from the definition of $M$ and
the consequence of Lemma~\ref{lem:cfs} stated above, provided
$C_2$ is sufficiently large. 
Thus Lemma~\ref{lem:cfs} yields a red copy of $K_3^{(3)}(s)$
whose three vertex classes are contained in $X$, $B'$ and $C'$,
respectively. Since these sets lie in three pairwise disjoint intervals,
arranging the classes from left to right gives a red ordered copy,
contradicting our assumption.
\end{proof}

\section{Proof of Theorem \ref{thm:main}}

\begin{proof}
The case $n=1$ follows by taking $N=t$: if there is no red triple,
all triples are blue.
Assume now that $n\ge 2$, and choose
  $\rho=\frac{1}{128}n^{-1/(12d+2)} $.
Then $0<\rho<1/128$. Let $u_1,\ldots,u_t$ be a weak $d$-degeneracy ordering of
$H$. For each $u_i$, define
\[
  r(u_i)=|\{\{u_i,u_j,u_k\}\in E(H):i<j<k\}|.
\]

Set $\delta_0=\rho^{3d}/(16d^2)$ and
\[
  M=
  2^{C_3\left(\delta_0^{-2}n^{3/2}
  +n^2\log\frac{1}{1-8\rho}
  +n\log\frac{1}{\delta_0}\right)},
\]
where $C_3\ge C_2$ is a sufficiently large absolute constant for the application of
Lemma~\ref{lem:bad} below.

Let
\[
  N=\sum_{v\in V(H)}16\rho^{-d}(r(v)+1)M.
\]
Partition the ordered vertex set $[N]$ of $K_N^{(3)}$ into consecutive intervals
$I_v$ for $v\in V(H)$ in the original order of $H$, with
\[
  |I_v|=16\rho^{-d}(r(v)+1)M.
\]

Consider any red-blue coloring of the triples. If there is a red ordered copy
of $K^{(3)}_3(n)$, we are done. Hence assume that there is no such red copy.
We embed the vertices in the degeneracy order $u_1,\ldots,u_t$ of $H$, while always
mapping each vertex $v$ into its prescribed interval $I_v$.
We write $f$ for
the resulting partial embedding.
This guarantees that
the final copy is ordered according to the original order of $H$.

After embedding $u_1,\ldots,u_i$, for each unembedded vertex $v$ we maintain a
candidate set $C_v^i\subseteq I_v$; and for each pair $y,z$ of unembedded vertices, we maintain a bipartite
graph $P_{y,z}^i\subseteq C_y^i\times C_z^i$. For $w\in C_y^i$,
write
\[
  N_{P_{y,z}^i}(w)=\{q\in C_z^i:wq\in E(P_{y,z}^i)\},
  \qquad
  \deg_{P_{y,z}^i}(w)=|N_{P_{y,z}^i}(w)|.
\]

For an unembedded vertex $v$, define
\[
  a_i(v)=|\{\{x,x',v\}\in E(H):x,x'\in \{u_1,\ldots,u_i\}\}|.
\]
For two unembedded vertices $y,z$, define
\[
  b_i(y,z)=|\{x\in \{u_1,\ldots,u_i\}:\{x,y,z\}\in E(H)\}|.
\]
Since $u_1,\ldots,u_t$ is a weak $d$-degeneracy ordering, we have
$a_i(v)\le d$ and $b_i(y,z)\le d$.

Set $\eta_v^i :=\frac{\rho^{d-a_i(v)}}{4d}$. We maintain the following invariants.

\begin{enumerate}
\item[(i)] For every unembedded $v$, $|C_v^i|\ge \rho^{a_i(v)}|I_v|$.

\item[(ii)] For every pair $y,z$ of unembedded vertices, $P_{y,z}^i$ is
bi-$(\eta_y^i,\eta_z^i,\rho^{b_i(y,z)})$-dense.

\item[(iii)] Compatibility: every edge whose vertices have all been
embedded is mapped to a blue triple. If $x$ is embedded and $\{x,y,z\}\in E(H)$ with $y,z$
unembedded, then every edge of $P_{y,z}^i$ forms a blue triple with $f(x)$; if
$x,x'$ are embedded and $\{x,x',v\}\in E(H)$ with $v$ unembedded, then for every $w\in C_v^i$,
the triple $\{f(x),f(x'),w\}$ is blue.

\end{enumerate}

Initially $C_v^0=I_v$ and $P_{y,z}^0$ is the complete bipartite graph between
$I_y$ and $I_z$, so the invariants hold.

Assume the invariants hold after step $i$, and let $v=u_{i+1}$ be the next
vertex. For every unembedded $z$ with $b_i(v,z)>0$, we require a candidate
$w\in C_v^i$ to satisfy
\begin{align}\label{candt-W}
  \deg_{P_{v,z}^i}(w)\ge \rho^{b_i(v,z)}|C_z^i|.
\end{align}

For a fixed such $z$, let $S$ denote the set of vertices $w\in C_v^i$
violating this inequality. If $S=\varnothing$, then clearly
$|S|<\eta_v^i|C_v^i|$. Otherwise, by the definition of $S$, 
\[
  e_{P_{v,z}^i}(S,C_z^i)=\sum_{w\in S}\deg_{P_{v,z}^i}(w)
  < \rho^{b_i(v,z)}|S||C_z^i|.
\]
Therefore the bi-$(\eta_v^i,\eta_z^i,\rho^{b_i(v,z)})$-density of $P_{v,z}^i$ forces $|S|<\eta_v^i|C_v^i|$.

Let $W\subseteq C_v^i$ be the set of candidates satisfying \eqref{candt-W} for every unembedded $z$ with $b_i(v,z)>0$.
By the weak $d$-degeneracy at $v$, there are at most $d$ unembedded vertices $z$ with $b_i(v,z)>0$. Hence
\[
  |W|\ge (1-d\eta_v^i)|C_v^i|
  \ge \frac34 |C_v^i|
  \ge \frac34 \rho^d |I_v|
  \ge 12(r(v)+1)M .
\]

Now consider an edge $\{v,y,z\}\in E(H)$, where $y,z$ are unembedded.
Let $B_{y,z}$ be the set of vertices $w\in W$ for which
$P_{y,z}^i\cap L_{\blue}(w)$ is not
bi-$(\eta_y^i,\eta_z^i,\rho^{b_i(y,z)+1})$-dense on
$C_y^i\times C_z^i$.

By invariant (i), the interval-size assumption, and the bounds
$a_i(\cdot),b_i(\cdot)\le d$, we have
\[
  |C_y^i|,|C_z^i|\ge M
  \quad\text{and}\quad
  \rho^{b_i(y,z)}\eta_y^i\eta_z^i
  \ge
  \rho^d\left(\frac{\rho^d}{4d}\right)^2
  =
  \delta_0 .
\]
Thus Lemma~\ref{lem:bad}, applied with
\[
  A=W,\qquad B=C_y^i,\qquad C=C_z^i,\qquad
  P=P_{y,z}^i,\qquad \sigma=\rho^{b_i(y,z)},\qquad s=n,
\]
gives $|B_{y,z}|<M$. Indeed, its threshold is decreasing in $\delta$,
and the actual parameter $\rho^{b_i(y,z)}\eta_y^i\eta_z^i$ is at least
$\delta_0$. Thus the threshold is at most $M$ by the choice of $C_3$.

Since there are exactly $r(v)$ edges of the form $\{v,y,z\}$ with
$y,z$ unembedded, and each corresponding bad set has size less than $M$,
\[
  \left|
  W\setminus
  \bigcup_{\{v,y,z\}\in E(H),\ y,z\text{ unembedded}} B_{y,z}
  \right|
  \ge
  |W|-r(v)M
  \ge
  12(r(v)+1)M-r(v)M
  >0 .
\]
Choose $w$ from this set and set $f(v)=w$.

We now update candidates. For an unembedded vertex $z$, set
\[
  C_z^{i+1}=
  \begin{cases}
  N_{P_{v,z}^i}(w),& b_i(v,z)>0,\\
  C_z^i,& b_i(v,z)=0.
  \end{cases}
\]
Then $|C_z^{i+1}|\ge \rho^{b_i(v,z)}|C_z^i|$.
Since $a_{i+1}(z)=a_i(z)+b_i(v,z)$, invariant (i) follows.

For an unembedded pair $y,z$, define
\[
  P_{y,z}^{i+1}=
  \begin{cases}
  \left(P_{y,z}^i\cap L_{\blue}(w)\right)[C_y^{i+1},C_z^{i+1}],
    & \{v,y,z\}\in E(H),\\
  P_{y,z}^i[C_y^{i+1},C_z^{i+1}],
    & \{v,y,z\}\notin E(H).
  \end{cases}
\]
Here $Q[A,B]$ denotes the bipartite subgraph of $Q$ induced by
$A$ and $B$, and $P_{y,z}^i\cap L_{\blue}(w)$ consists of those
edges of $P_{y,z}^i$ that form a blue triple with $w$.

We now verify that invariant (ii) is preserved. By the definition of $\eta$
and the update of the candidate sets, for each unembedded $y$ we have
\[
\eta_y^{i+1}
  =
  \frac{\rho^{d-a_{i+1}(y)}}{4d}
  =
  \frac{\rho^{d-a_i(y)-b_i(v,y)}}{4d}
  =
  \rho^{-b_i(v,y)}\eta_y^i,
  \qquad
  |C_y^{i+1}|\ge \rho^{b_i(v,y)}|C_y^i|,
\]
and similarly for $z$. Hence any subset of $C_y^{i+1}$ of size at least
$\eta_y^{i+1}|C_y^{i+1}|$ has size at least $\eta_y^i|C_y^i|$, and the same
holds on the $z$-side.

If $\{v,y,z\}\in E(H)$, then $w\notin B_{y,z}$ gives the required bi-density
for the restriction
  $P_{y,z}^{i+1}$
with  density parameter
\[
  \rho^{b_i(y,z)+1}=\rho^{b_{i+1}(y,z)}.
\]
If $\{v,y,z\}\notin E(H)$, then the required bi-density for $P_{y,z}^{i+1}$
follows from the induction hypothesis for $P_{y,z}^i$, with density parameter
\[
  \rho^{b_i(y,z)}=\rho^{b_{i+1}(y,z)}.
\]
Thus invariant (ii) is preserved.

Finally, compatibility (iii) is preserved. Previously imposed compatibility
conditions remain valid after restricting candidate sets and bipartite graphs.
The only new conditions involve the newly embedded vertex $v$. Edges of the
forms $\{x,x',v\}$, $\{x,v,z\}$, and $\{v,y,z\}$ are guaranteed, respectively,
by $w\in C_v^i$, by the update $C_z^{i+1}=N_{P_{v,z}^i}(w)$ together with
compatibility at step $i$, and by the restriction
$P_{y,z}^{i+1}\subseteq L_{\blue}(w)$. This completes the induction step.

After all vertices have been embedded, compatibility ensures that every edge of
$H$ is mapped to a blue triple. Since $f(v)\in I_v$ for every $v$ and the
intervals $I_v$ are arranged in the order of $H$, this copy is ordered.

It remains to estimate $N$. Each edge $\{u_a,u_b,u_c\}$ with
$a<b<c$ contributes once to $\sum_v r(v)$ and is counted in the
weak-degeneracy condition at both $u_b$ and $u_c$. Hence
$2e(H)\le dt$, and therefore
\[
  \sum_v(r(v)+1)=e(H)+t\le\left(1+\frac d2\right)t.
\]
Consequently, by the definition of $N$,
\[
  N=16\rho^{-d}M\sum_v(r(v)+1)
  \le16\left(1+\frac d2\right)t\rho^{-d}M.
\]

We now estimate $M$. Substituting
$\delta_0=\rho^{3d}/(16d^2)$ and
$\rho=128^{-1}n^{-1/(12d+2)}$, we obtain
$
  \delta_0^{-2}n^{3/2}
  =256d^4\rho^{-6d}n^{3/2}
  =O_d\left(n^{2-\frac1{12d+2}}\right).
$
Moreover, $\log(1/(1-x))\le2x$ for $0\le x\le1/2$ gives
$
  n^2\log\frac1{1-8\rho}
  =O(n^2\rho)
  =O\left(n^{2-\frac1{12d+2}}\right),
$
while $n\log(1/\delta_0)=O_d(n\log n)
=O_d(n^{2-1/(12d+2)})$.
Thus the definition of $M$ implies $\log M=O_d(n^{2-1/(12d+2)})$. 

Combining this with the bound on $N$ and $\log(1/\rho)=O_d(\log n)$, we conclude that
\[
  \log\frac Nt
  \le\log\left(16\left(1+\frac d2\right)\right)
     +d\log\frac1\rho+\log M
  =O_d\left(n^{2-\frac1{12d+2}}\right).
\]
Consequently, $N\le t\,2^{C_d n^{2-1/(12d+2)}}$ for a suitable
constant $C_d>0$. Since $1/(12d+2)\ge1/(14d)$ for $d\ge1$,
the theorem follows with the absolute constant $c=1/14$.
\end{proof}

\section{Standard degeneracy is not sufficient}

\begin{proof}[Proof of Theorem~\ref{main-2}]
We construct a red--blue coloring of the triples of $[N]$ containing neither
a blue ordered copy of $F_m$ nor a red ordered copy of $K^{(3)}_3(n)$.
Fix
  $N=\left\lceil2^{n^2/100}\right\rceil$,
  $m=\left\lfloor2^{n/10}\right\rfloor$ and
  $p=2^{-n/20}$.
We will take $F=F_m$, where the ordered $3$-graph $F_m$ is defined as follows. Its vertices are
\[
  v_1,\ldots,v_m,\quad u_{ij}\quad (1\le i<j\le m),
\]
ordered by
\[
  v_1<\cdots<v_m<u_{12}<u_{13}<\cdots<u_{m-1,m},
\]
where the vertices $u_{ij}$ are ordered lexicographically. Its edges are
  $\{v_i,v_j,u_{ij}\}$ where $1\le i<j\le m$.
The displayed ordering witnesses standard $1$-degeneracy: no edge
containing $v_i$ has both other vertices earlier in the order, while
each $u_{ij}$ belongs to exactly one edge.

We now construct a coloring of the triples of $[N]$. First take a random
red-blue coloring $\chi$ of the pairs of $[N]$, where each pair is coloured
red independently with probability $p$ and blue otherwise. We then use $\chi$
to define a red-blue coloring of the complete ordered $3$-graph on $[N]$:
for every triple $x<y<z$, colour $xyz$ red if $xy$ is red under $\chi$, and
blue otherwise.

We claim that, with positive probability, the auxiliary pair-coloring $\chi$
contains neither a red $K_{n,n}$ nor a blue $K_m$. Indeed, using
$N\le2^{n^2/100+1}$, $2^{n/10-1}\le m\le2^{n/10}$ and
$\binom m2\ge m^2/4$, a union bound gives
\[
\begin{aligned}
  \binom{N}{n}^{2}p^{n^2}
  +\binom{N}{m}(1-p)^{\binom m2}
  &\le N^{2n}p^{n^2}+N^m e^{-p\binom m2}\\
  &\le 2^{-3n^3/100+2n}
  +2^{(n^2/100+1)2^{n/10}-2^{3n/20-4}}\\
  &=o(1),
\end{aligned}
\]
since $2^{n/20}/n^2\to\infty$.
Thus, for all sufficiently large $n$, such an auxiliary pair-coloring $\chi$
exists.

Fix such a $\chi$, and consider the induced coloring of the complete ordered
$3$-graph on $[N]$. There is no red ordered copy of $K_3^{(3)}(n)$:
Suppose that there is a red ordered copy of $K_3^{(3)}(n)$
with vertex classes $A<B<C$, each of size $n$. Fix $c\in C$.
For every $a\in A$ and $b\in B$, the triple $abc$ is red and
$a<b<c$, so the definition of the triple-coloring implies
that $\chi(ab)$ is red. Thus all pairs between $A$ and $B$
are red under $\chi$, contradicting the absence of a red
$K_{n,n}$.

There is also no blue ordered copy of $F_m$. Indeed, suppose such a copy exists,
and let $x_i$ be the image of $v_i$. Since the copy is ordered,
\[
  x_1<\cdots<x_m,
\]
and each image of some $u_{ij}$ lies to the right of all the $x_i$. For every
$i<j$, the triple corresponding to $\{v_i,v_j,u_{ij}\}$ is blue, so the pair
$x_ix_j$ is blue under $\chi$. Thus $x_1,\ldots,x_m$ span a blue $K_m$ in the
auxiliary pair-coloring, again a contradiction.

Therefore the complete ordered $3$-graph on $[N]$ admits a red-blue coloring
with no blue ordered copy of $F_m$ and no red ordered copy of $K_3^{(3)}(n)$.
Since $|V(F_m)|=m+\binom m2\le m^2\le2^{n/5}$ and
$r_<(F_m,K_3^{(3)}(n))>N\ge2^{n^2/100}$, this proves the theorem.
\end{proof}

\section*{Declaration on the Use of AI}

The authors provided the key ideas underlying the proofs of Theorems~\ref{thm:main} and~\ref{main-2}. Generative AI tools were used to assist in discussing proof strategies, checking proofs, and improving the exposition. All mathematical arguments, results, and conclusions were reviewed and verified by the authors.


\begin{thebibliography}{99}

\bibitem{balko25}
M. Balko, A survey on ordered Ramsey numbers, preprint, arXiv:2502.02155 [math.CO], 2025.

\bibitem{bckk13}
M. Balko, J. Cibulka, K. Kr\'al and J. Kyn\v{c}l, Ramsey numbers of ordered graphs, Electron. J. Combin. {\bf 27} (2020), no.~1, Paper No. 1.16, 32 pp.

\bibitem{bjv16}
M. Balko, V. Jel\'inek and P. Valtr, On ordered Ramsey numbers of bounded-degree graphs, J. Combin. Theory Ser. B {\bf 134} (2019), 179--202.

\bibitem{BV}
M. Balko and M. Vizer, On ordered Ramsey numbers of tripartite 3-uniform hypergraphs, SIAM J. Discrete Math. {\bf 36} (2022), no.~1, 214--228.

\bibitem{crst83}
V. Chv\'atal, V. R\"odl, E. Szemer\'edi and W. T. Trotter, Jr., The Ramsey number of a graph with bounded maximum degree, J. Combin. Theory Ser. B {\bf 34} (1983), no.~3, 239--243.

\bibitem{clfs17}
D. Conlon, J. Fox, C. Lee and B. Sudakov, Ordered Ramsey numbers, J. Combin. Theory Ser. B {\bf 122} (2017), 353--383.

\bibitem{cfs09}
D. Conlon, J. Fox and B. Sudakov, Ramsey numbers of sparse hypergraphs, Random Structures Algorithms {\bf 35} (2009), no.~1, 1--14.

\bibitem{cfs11}
D. Conlon, J. Fox and B. Sudakov, Large almost monochromatic subsets in hypergraphs, Israel J. Math. {\bf 181} (2011), 423--432.

\bibitem{CFS}
D. Conlon, J. Fox and B. Sudakov, Erd\H{o}s---Hajnal-type theorems in hypergraphs, J. Combin. Theory Ser. B {\bf 102} (2012), no.~5, 1142--1154.

\bibitem{cfsSurvey}
D. Conlon, J. Fox and B. Sudakov, Recent developments in graph Ramsey theory, in {\it Surveys in combinatorics 2015}, 49--118, London Math. Soc. Lecture Note Ser., 424, Cambridge Univ. Press, Cambridge.

\bibitem{cnko08}
O. Cooley, N. Fountoulakis, D. K\"uhn and D. Osthus, 3-uniform hypergraphs of bounded degree have linear Ramsey numbers, J. Combin. Theory Ser. B {\bf 98} (2008), no.~3, 484--505.

\bibitem{cnko09}
O. Cooley, N. Fountoulakis, D. K\"uhn and D. Osthus, Embeddings and Ramsey numbers of sparse $k$-uniform hypergraphs, Combinatorica {\bf 29} (2009), no.~3, 263--297.

\bibitem{erdosSzekeres35}
P. Erd\H{o}s and G. Szekeres, A combinatorial problem in geometry, Compositio Math. {\bf 2} (1935), 463--470.

\bibitem{foxHe}
J. Fox and X. He, Independent sets in hypergraphs with a forbidden link, Proc. Lond. Math. Soc. (3) {\bf 123} (2021), no.~4, 384--409.

\bibitem{fpss12}
J. Fox, J. Pach, B. Sudakov and A. Suk, Erd\H{o}s-Szekeres-type theorems for monotone paths and convex bodies, Proc. Lond. Math. Soc. (3) {\bf 105} (2012), no.~5, 953--982.

\bibitem{moshShap14}
G. Moshkovitz and A. Shapira, Ramsey theory, integer partitions and a new proof of the Erd\H{o}s-Szekeres theorem, Adv. Math. {\bf 262} (2014), 1107--1129.

\bibitem{nsrs08}
B. Nagle, S. Olsen, V. R\"odl and M. Schacht, On the Ramsey number of sparse 3-graphs, Graphs Combin. {\bf 24} (2008), no.~3, 205--228.

\end{thebibliography}
\end{document}